\documentclass[12pt, a4paper,oneside]{amsart} 
\usepackage{amssymb, amsfonts, mathrsfs}
\usepackage{url}
\usepackage{textcomp}
\usepackage[bookmarks]{hyperref}

\usepackage[margin=1in]{geometry}

\usepackage{glossaries}
\usepackage[bookmarks]{hyperref} 
\usepackage[maxbibnames=99]{biblatex}
\bibliography{references.bib}
\DeclareFieldFormat*{url}{}
\DeclareFieldFormat[unpublished]{url}{\mkbibacro{URL}\addcolon\space\url{#1}}
\DeclareFieldFormat*{urldate}{}
\DeclareFieldFormat[unpublished]{urldate}{\mkbibparens{\bibstring{urlseen}\space#1}}

\usepackage{dirtytalk}

\usepackage{lmodern} 

\usepackage[textwidth=1in,backgroundcolor=white,linecolor=white,bordercolor=white,textcolor=red]{todonotes}

\usepackage{amsmath}
\usepackage{amsthm}
\usepackage{mathtools}  
\mathtoolsset{showonlyrefs}  

\usepackage{tkz-euclide}
\usepackage{tikz}
\usetikzlibrary{patterns}
\usepackage{thmtools, thm-restate}
\theoremstyle{definition}
\newtheorem{definition}{Definition}[section]

\newtheorem{remark}[definition]{Remark}

\theoremstyle{plain}

\newtheorem{theorem}[definition]{Theorem}
\newtheorem{proposition}[definition]{Proposition}

\newtheorem{lemma}[definition]{Lemma}

\renewcommand\P{\mathbb{P}}

\newcommand\R{\mathbb{R}}

\newcommand\HH{\mathbb{H}}
\newcommand\Sp{\mathbb{S}}

\newcommand\T{\mathbb{T}}

\newcommand\FF{\mathcal{F}}

\newcommand\ti[1]{\widetilde{#1}}
\newcommand\tq{\; | \;}

\newcommand\SL{\mathrm{SL}}

\newcommand\Hom{\mathrm{Hom}}

\newcommand{\ps}[1]{\left< #1 \right>}

\newcommand\hol{\mathrm{hol}}
\newcommand\binf{\partial_{\infty}}
\newcommand\holder{\text{Hölder}}
\newcommand\PSL{\text{PSL}}

\usepackage{mathtools}
\DeclarePairedDelimiter\abs{\lvert}{\rvert}
\DeclarePairedDelimiter\norm{\lVert}{\rVert}

\usepackage{dirtytalk}

\title[Gap between Lyapunov exponents]{Gap between Lyapunov exponents for Hitchin representations}

\author{Matteo Costantini \and Florestan Martin-Baillon}
\begin{document}

\maketitle{}

\begin{abstract}
We study Lyapunov exponents for flat bundles over hyperbolic curves defined via parallel transport over the geodesic flow. We consider them as invariants on the space of Hitchin representations and show that there is a gap between any two consecutive Lyapunov exponents. Moreover we characterize the uniformizing representation of the Riemann surface as the one with the extremal gaps.

The strategy of the proof is to relate Lyapunov exponents in the case of Anosov representations to other invariants, where the gap result is already available or where we can directly show it. In particular, firstly we relate Lyapunov exponents to a foliated Lyapunov exponent associated to a foliation H\"older isomorphic to the unstable foliation on the unitary tangent bundle of a Riemann surface. Secondly, we relate them to the renormalized intersection product in the setting of the thermodynamic formalism
developed by
Bridgeman, Canary, Labourie and Sambarino.
\end{abstract}

\setcounter{tocdepth}{1}
\tableofcontents{}

\section{Introduction}
\label{sec:introduction}

Lyapunov exponents are characteristic numbers associated to the dynamics of trajectories of a dynamical system. 
The interest for these invariants in this paper's context originated from the study of the dynamical properties of billiard trajectories on polygonal billiards and trajectories of wind-tree models
for the diffusion of gas molecules.
Both these settings can be studied by restating the problem in terms of properties of the geodesic flow on a flat surface,
i.e. a topological surface equipped with a flat metric with finitely many conical singularities (see the survey \cite{zorichFlatSurfaces2006}). 

It turned out that, in order to study properties of a special flat surface,
it is convenient to study properties of the associated family given by the deforming the flat surface.
The Lyapunov exponents of the original problem for a special flat surface
can be identified with the ones associated to the flat cohomology bundle
over the flat surface associated family. 
It is at this point that the work of Eskin-Kontsevich-Zorich
\cite{eskinSumLyapunovExponents2014}
allowed to compute the sum of the positive Lyapunov exponents in this setting
by relating it to the algebraic degree of a holomorphic vector bundle.

Consequent works
generalized the relation between Lyapunov exponents and degrees
of holomorphic bundles in the case of special flat bundles
coming from families of curves over ball quotients
\cite{kappesLyapunovSpectrumBall2016},
family of K3 surfaces
\cite{filipFamiliesK3Surfaces2018},
and more generally in the case of any flat bundle over a Riemann surface
\cite{eskinLowerBoundsLyapunov2018}.
Daniel-Deroin \cite{danielLyapunovExponentsBrownian2017}
generalized even more this relation to the case of flat bundles over K\"ahler manifolds.
The first author
\cite{costantiniLyapunovExponentsHolomorphic2020}
refined the relation of
\cite{eskinLowerBoundsLyapunov2018}
and proposed to study the Lyapunov exponents as functions on representation varieties
and conjectured an inequality for the gap between Lyapunov
exponents for Hitchin representations.

Hitchin representations are the central objects
of study
in \emph{Higher Teichmüller Theory},
which generalizes the theory of Fuchsian and Kleinian
groups to Lie groups of rank $ \ge 2 $ (see \autoref{sub:anosov_representations}).
They were introduced by Hitchin
\cite{hitchinLieGroupsTeichmuller1992}
as representations in the connected components of character varieties
containing an embedding of the Teichmüller space.
Labourie
\cite{labourieAnosovFlowsSurface2006}
showed that Hitchin representations possess a rich dynamical
structure and in particular they are faithful
and discrete.

Generalizing the rank one case,
various asymptotic quantities
describing the geometry
of a representation have been
studied:
the orbital counting problem
\cite{sambarinoQuantitativePropertiesConvex2014},
critical exponent and
entropies
\cite{potrieEigenvaluesEntropyHitchin2017},
Hausdorff dimension of limit sets
\cite{pozzettiConformalityRobustClass2019}
\cite{glorieuxHausdorffDimensionLimit2019}.
A defining feature of representations in higher rank
is that the relevant notion of \say{size} of a matrix is 
not just a number, the norm,
but the collection of all the singular values,
which form a vector called the
\emph{Cartan projection}
which lives in the Cartan subspace.
The asymptotic geometry of
the Cartan projection of the
image of the representation
is a central subject of investigations.
In the present work we study
the Lyapunov exponents of a Hitchin representation
with respect to a hyperbolic metric on a surface. These characteristic numbers are 
quantities that measure the asymptotic growth of the norm of vectors under parallel transport in the flat bundle associated to the representation, where the parallel transport happens over the geodesic flow defined by the hyperbolic metric.
The Lyapunov exponents define a vector in the Cartan subspace which reflects asymptotic properties of the representation
with respect to the hyperbolic metric.

Let $ S $ be a compact surface of genus $ g \ge 2 $
and $ X $ be a structure of Riemann surface on $ S $.
Given a representation
$ \rho : \pi_1(S) \to \SL(d, \R) $,
we can define the Lyapunov exponents
$ \lambda_{1}(X,\rho) \ge \dots \ge \lambda_{d}(X,\rho) $ 
of $ \rho $ with respect to $ X $ (see \autoref{subsec:Oseledets}).

Computer experiments made by the first author hinted of a gap between the first and the second Lyapunov exponents for Hitchin representations.  In this work, we can show the following more general statement.
\begin{theorem}
	\label{th:main}
	Let $X$ be a structure of Riemann surface on a compact  surface $ S $ of genus $ g \ge 2 $ and
	$ \rho : \pi_1(X) \to \SL (d,\R) $
	be a Hitchin representation.
	Then it holds 
	\begin{equation}
		\lambda_{i}(X,\rho) - \lambda_{i+1}(X,\rho)
		\ge 1
	\end{equation}
 for every $ i = 1, \dots, d-1 $.
 
 Moreover the bound is attained for every  $ i = 1, \dots, d-1 $ if and only if $\rho$ is conjugated
 to the image of the Fuchsian representation
uniformizing $ X $
under the irreducible representation
$ \SL(2, \R) \to \SL(d, \R)$.
\end{theorem}

The previous result follows from a more general statement regarding
a specific class of Anosov representations,
the (1,1,2)-hyperconvex Anosov representations.
They are Anosov representations whose action on their limit set
exhibit a form of asymptotic conformality
(see \autoref{sub:anosov_representations}).

\begin{theorem}
	\label{th:main_anosov}
	Let $X$ be a structure of Riemann surface on a  compact surface $ S $ of genus $ g \ge 2 $ and let 
	$ \rho : \pi_1(X) \to \SL(d, \R) $
	be a (1,1,2)-hyperconvex Anosov representation.
	Then it holds
	\begin{equation}
		\lambda_{1}(X,\rho) - \lambda_{2}(X,\rho)
		\ge 1
		.
	\end{equation}
\end{theorem}

We give two proofs of the previous result, both being a consequence of the relation we show between Lyapunov exponents and other characteristic numbers (see \autoref{thm:relations}).
The first proof relies on the study of 
a deformation of the weak unstable
foliation of the geodesic flow on the unitary tangent
bundle of $X$ and it is a consequence of a more general statement about the transverse Lyapunov exponent of this foliation (see
\autoref{sec:foliated_lyapunov_exponent}).
The second proof is more concise but it uses the
machinery of the
thermodynamic formalism
for Anosov representations
developed in 
\cite{bridgemanSimpleRootFlows2017} and \cite{bridgemanPressureMetricAnosov2015} 
 (see \autoref{sec:thermo}). It is however only in this context where we can show that the equalities of \autoref{th:main} characterizes the uniformizing representation.

These two approaches fit naturally
in the perspective that
Tholozan
developed in his note
\cite{tholozanTEICHMULLERGEOMETRYHIGHEST}.
He explains that there is a correspondence
between
\emph{Anosov actions on the circle},
\emph{deformations of the weak unstable foliation
of the geodesic flow},
and
\emph{reparametrizations of the geodesic flow}.
Furthermore this correspondence preserves
the
(appropriately defined)
periods of each object.
The two proofs we provide here
consist of considering
the Anosov action on the circle
given by an Anosov representation
and interpreting it
in one case as a deformation
of the weak unstable foliation
and in the other case
as a reparametrization of the geodesic flow.
It is however interesting to notice that
the two proofs are not a simple translation
of each other.
It is probable that there exists
a third proof using directly the Anosov action on the circle
which would use a random walk discretizing the geodesic flow
and an adapation of  Ledrappier formula
\cite{ledrappierRelationEntreEntropie}
in this context.
\par
\subsection*{Organization}
In \autoref{sec:foliated_lyapunov_exponent} we define the transverse Lyapunov exponent associated to a foliation H\"older isomorphic to the unstable foliation on $T^1X$ and prove the main bound for this quantity.

In \autoref{sec:thermo} we recall the setting of the thermodynamic formalism
developed by
Bridgeman, Canary, Labourie and Sambarino
in
\cite{bridgemanSimpleRootFlows2017} and \cite{bridgemanPressureMetricAnosov2015}
(see also \cite{bridgemanHessianHausdorffDimension2020}), in particular the main bound about the renormalized intersection.

In \autoref{sec:lyapexp} we recall the main definitions of Lyapunov exponents via Oseledets Theorem and the properties of Anosov and Hitchin representations. We also show some implications that being Anosov has on the Lyapunov exponents.

In \autoref{sec:relationships_between_the_different_lyapunov_exponents} we show the relation between Lyapunov exponents, the foliated Lyapunov exponent and the theormodynamic formalism. We finally prove the main results \autoref{th:main} and \autoref{th:main_anosov}.

\par
\subsection*{Acknowledgments}
We would like to thank Bertrand Deroin, Jérémy Daniel and Nicolas Tholozan for key insights and precious remarks. We would also like to thank the University of Frankfurt and CIRM for their hospitality, since an important part of this work has been developed there.

The first author has been supported by the DFG Research Training Group 2553.

\section{Transverse Lyapunov exponent}
\label{sec:foliated_lyapunov_exponent}

In this section we provide a general setting to study a new invariant measuring the growth of the holonomy of a reparametrization of the weak unstable foliation of the geodesic flow on a Riemann surface. We call this invariant the \emph{transverse Lyapunov exponent} and we will relate it to the usual Lyapunov exponents in the case where the reparametrization of the foliation is induced by a $(1,1,2)$-hyperconvex Anosov representation.

\subsection{Setup and notation}
\label{subsec:notation_transverse}

Let $S$ be a compact surface and $X$ be a Riemann surface structure on $S$. 
We denote by $ T^{1} X $ the unitary tangent bundle of $ X $, by $\pi:T^1X\to X$ the projection and by $ (\Psi_{t}) $ the geodesic flow on $T^1X$. We also consider the 
the weak unstable foliation $\FF_{u}$ of the geodesic flow on $T^{1} X$.
We finally equip $ T^{1} X $ with the Liouville volume form $ v_{L} $, normalized to be a probability measure.

From now on we say that a function is $C^{k+\holder}$ if it has  continuous derivatives up through order $k$ and such that the $k$-th partial derivatives are H\"older continuous for some exponent $\beta$, where $0 < \beta \leq 1$.

Let now $M$ be a $ C^{1+\holder} $-manifold which is H\"older isomorphic to $T^{1} X $. Let us call $ \Xi:T^{1} X\overset{\cong}{\rightarrow} M $ the isomorphism and assume that $ \Xi$ is $C^1$ along the leaves of $\FF_{u}$. Denote by $\FF_{u}^M$ the foliation on $M$ induced by $ \Xi$ and assume that this foliation is  $ C^{1+\holder} $. Denote also by
$ \Psi^M_t = \Xi \circ \Psi_t \circ \Xi^{-1} $
the flow induced on  $ M$. 
Note that 
$\FF_{u}^M$ is by definition  the weak unstable foliation of $ \Psi^M_t$, and this flow  is Hölder and $C^1$ along the leaves
of $ \FF_{u}^M $.

\begin{remark}
    This setting may seem arbitrary
    but it is
     exactly the setting we are in
    when we deform the conformal class
    of the weak unstable foliation
    of $ \Psi_t$
    on
    $T^1 X$,
    c.f. \cite{tholozanTEICHMULLERGEOMETRYHIGHEST}.
\end{remark}

Let us finally define a volume form $\nu$ on $M$ in the following way. 
Since the foliation is transversly oriented, we can fix an arbitrary
$1$-form $\alpha $ on $ M$ which defines the foliation,
meaning that its kernel defines the tangent bundle of $\FF_{u}^M$.
The form $\alpha$ induces a metric on the normal bundle of $ \FF_{u}^M $ and
 we define a volume form by $ \nu = \alpha \wedge \pi^{M,*}\omega_{P} $,
where $ \omega_{P} $ is the Poincaré metric on the surface $ X $ and $\pi^M:M\to X $ the projection $\pi^M:=\pi\circ\Xi^{-1}$.
Up to normalizing $ \alpha $
we can assume that $ \nu $ induces a probability measure. 

\begin{remark}
\label{rem:leafwise measure}
Note that that $ T^{1} X $ is equipped with the Liouville measure
$ v_{L} $ 
and $ M $ with the volume form $ \nu $.
The measures $ \nu $ and $ \Xi_{*} v_{L} $ 
are in general mutually singular, but the disintegrations of $ \nu $ and $ \Xi_{*} v_{L} $ along the leaves
of $\FF_{u}^M $ 
are both equal to the Poincaré leafwise measure $\pi^{M,*}\omega_P$.
\end{remark}

\subsection{The definition of the transverse Lyapunov exponent}

We will now define an invariant of the flow
$ \Psi^M_t $ and the foliation $\FF_{u}^M $,
which we call the \emph{transverse Lyapunov exponent}, which measures the asymptotic growth of the norm transverse to $\FF_{u}^M $  of vectors under the flow.

In order to define these characteristic numbers, we would like to  apply directly
ergodic theory machinery, but the subtlety here is
that the measure $ \nu $
is not preserved by the flow.
We hence use the Liouville measure
as an accessory, and then
show that we can compute
the foliated Lyapunov exponent
for almost every point
with respect to $ \nu $.

Given a path $ c:[0,t] \to M $ contained in a leaf, we define
$ \abs{D\hol(c ) }_\alpha $ as the norm of the derivative of the holonomy
of this path along the foliation $\FF_{u}^M $. More precisely, the derivative of the holonomy induces a linear map between
the one-dimensional fibers of the normal bundle $ N_{\FF_{u}^M ,c(0)} $ 
and $ N_{\FF_{u}^M ,c(t)} $, and we define
$ \abs{D\hol(c ) }_\alpha $ as the norm of this linear
map for the metric on $ N_{\FF_{u}^M } $ induced by $ \alpha $.
Note that
$ \abs{D \hol(c)}_\alpha $
depends on our choice of $ \alpha $.

The local expression of
$ \abs{D\hol(c) }_\alpha $ can be described in the following way.
Suppose that the image of $ c $
is contained in a chart $ V $
which trivializes the foliation, i.e.
$ V \simeq U \times I $
where $ U $ is a disk in $ \HH^{2} $
and $ I $ is an interval.
In this chart the measure
$ \nu $ can be written as 
$ f_\alpha(z,x) \omega_{P}(dz)dx $, so 
\begin{equation}
\label{eq:local_hol}
\abs{D\hol(c) }_\alpha
=f_\alpha(c(t))/f_\alpha(c(0))
.
\end{equation}

Let
$ \Psi^M_{[0,t]} (x) $
be the path
$ s \mapsto \Psi^M_s(x) $
defined on $ [0,t] $.

\begin{theorem}
	\label{th:def_chi}
	There exists a number $ \lambda_{T} $ such that for
	every leaf $ L $ of $ \FF_{u}^M $ and for
 $\pi^{M,*}\omega_P$
 almost every $ x \in L $
	we have
	\begin{equation}
		\lim_{t \to +\infty}
		\frac{1}{t}
		\log
		\left|
			D\hol\left(\Psi^M_{[0,t]}(x)\right)
		\right|_\alpha
		=
		\lambda_{T}
		.
	\end{equation}
\end{theorem}
We call the number $ \lambda_{T} $ the \emph{transverse Lyapunov exponent}. 

\begin{remark}
\label{rmk:recurrence}
The transverse Lyapunov exponent $ \lambda_{T} $ is independent of the transverse form $\alpha$. Indeed, since $M$ is compact, any two norms are uniformly bounded away from each other. 

\end{remark}

\begin{proof}
	For any $ x \in T^{1} X $,
	we set $ H_{t} (x) =
	\log
	\abs{
		D\hol(\Psi^M_{[0,t]}(\Xi(x)))
	}_\alpha
	$.
	This is a cocycle over
	the geodesic flow on $ T^{1} X $.
	As the geodesic flow on $ T^{1} X $ is ergodic with
	respect to the Liouville measure, there exists
	a number $ \lambda_{T} $ such that, for $ v_{L} $-almost
	every $ x \in T^{1} X $, we have
	\begin{equation}
		\lim_{t \to +\infty}
		\frac{1}{t}
		H_{t} (x)
		=
		\lambda_{T}
		.
	\end{equation}
	Let us denote by $ G $ the set of \say{good points}, that is
	those for which the above equality holds.
	We know that it is a set of full $ v_{L} $ measure,
	and it is invariant under the flow.
	We are going to show that for \emph{every} leaf $ L $
	of $ \FF_{u} $,
	the set $ G \cap L $ is of full measure. This implies the main statement, as $ \Xi $ maps leaves of $ \FF_{u} $ to  leaves
	of $ \FF_{u}^M $ and by \autoref{rem:leafwise measure} it maps full measure subsets to full measure
	subsets.
 
	First, it is clear that there exists one leaf $ L_{0} $ for which this is true.
	Then we are going to show that $ G $ is a set foliated
	by the \emph{strong stable} foliation $ \FF_{ss} $,
	that is the
	1-dimensional foliation
	such that a leaf through a point $ x $ is the set of $ y $'s
	such that $ d(\Psi_{t}(x), \Psi_{t}(y)) $
	goes to $ 0 $ exponentially fast. This will imply the result. Indeed, starting from the leaf
	$ L_{0} $ such that $ G \cap L_{0} $ is of full measure,
	consider another leaf $ L $ of $ \FF_{u} $. Consider
	the first return map of the horocyclic flow $ L_{0} \to L $.
	Because the strong stable and unstable foliation are transverse
	to each other, this is a smooth map and it is surjective.
	It sends the set of full measure $ L_{0} \cap G $ to a set of full
	measure which is included in $ L \cap G $.



	Let's finally show that the set $ G $ is foliated by the strong stable foliation.
	The holonomy of the foliation $ \FF_{u}^M $ is $ C^{1+\holder} $.
	In particular, the derivative of the holonomy along a geodesic
	of length $ 1 $ is \emph{uniformly} $ \beta$-Hölder,
	for some $ \beta > 0 $. Let us denote by  $ \norm{\cdot}_{C^{\beta}} $
    the $\beta$-Hölder norm
    \[\norm{f}_{C^{\beta}}:=\sup_{x,y\in T^1X}\frac{|f(x)-f(y)|}{d(x,y)^\beta}\]
	where $f:T^1X\to \R$ and $d(\cdot,\cdot)$ is a distance in $T^1X$ (any two distances are equivalent since $T^1X$ is compact).
	Then, by compactness of $ M$,
	there exists a constant $ C > 0 $ such that
	\begin{equation}
		\norm{
			x \mapsto
			H_{1} (x)
		}_{C^{\beta}}
	\le C.
	\end{equation}
    
	Let $ x $ and $ y $
	be in the same leaf of the
	strong stable foliation.
	We show now that $ x $ and $ y $
	belong to $G$.
	We denote $ z_{k} = \Psi_{k} (z) $.
	We have then
	\begin{align}
		\norm{H_{n}(x)
		-	
		H_{n}(y)}_{C^{\beta}}
		&\leq 
		\sum_{k=0}^{n-1}	
		\norm{H_{1}(x_{k})
		-	
		H_{1}(y_{k})}_{C^{\beta}}
		\le
		C
		\sum_{k=0}^{n-1}
		d(x_{k}, y_{k})^{\beta} \le
		C
		\sum_{k=0}^{\infty}
		d(x_{k}, y_{k})^{\beta} 
		,
	\end{align}
	and the last sum is convergent because
	$ x_{k} $ and $ y_{k} $ become exponentially
	close.
	This implies that $\lim_{n\to \infty} H_{n}(x)/n=\lim_{n\to \infty} H_{n}(y)/n$, which is what we wanted to show.
\end{proof}

The conclusion of the theorem
implies that
the limit defining $ \lambda_{T} $
exists for $ \nu $-almost every
$ x \in M $,
because $ \Xi_{*} v_{L} $ and $ \nu $
are absolutely continuous along the leaves.

The manifold
$ M $
is a foliated fiber bundle $\pi^M:M\to X$, where $\pi^M=\pi\circ\Xi^{-1}$. Hence,
the transverse Lyapunov exponent
can be expressed using the fibered structure.
The flow $ \Psi^M_t $
induces a projective transformation
between fibers
$ \Psi^M_{t,x} :
M_{\pi^M(x)}
\to
M_{\pi^M\left(\Psi^M_{t}(x)\right)}
$ 
for each $ x \in M $.
For each
$ \zeta \in
M_{\pi^M(x)}$,
we denote by
$ D_{\zeta} \Psi^M_{t,x} $
the derivative
of
$ \Psi^M_{t,x} $
at the point $ \zeta $.
Then we have
\begin{equation}
\label{eq:transverse}
\lambda_{T}=\lim_{t \to +\infty} \frac{1}{t}\log\norm{D_{\zeta}\Psi^M_{t,x}}
\end{equation}
for almost every $ x $ in every leaf,
where $ \zeta $ is the point in
$ M_{\pi^M(x)} $ corresponding
to $ x $
and for any choice of norm
on the tangent bundle to the fibers.
This follows from the
the fact that
the holonomy along
paths of the form
$ \Psi^M_{[0,t]}(x) $
is given by
the maps
$ \Psi^M_{t,x} $
between fibers
and the tangent bundle
to the fibers are identified
with the normal bundle
of the foliation.

\subsection{Bound on the transverse Lyapunov exponent}
\label{subsec:bound_on_the_foliated_lyapunov_exponent}

We will now show the main estimate for the transverse Lyapunov exponent.

\begin{theorem}
	\label{th:bound_chi_fol}
	The transverse Lyapunov exponent
	satisfies
	$ \lambda_{T} \le -1 $.
\end{theorem}

The previous theorem is a consequence of the
following lemma which explains
how the measure
$ \nu $
is transformed by the flow.

\begin{lemma}
\label{lemma:flowmeasure}
	The transformation of the measure $ \nu $ under the flow on $M$
	is given by 
	\begin{equation}
		(\Psi^M_{-t})_{*}
	\nu(dx)
	=
	e^{t}
	\abs{D \hol \Psi^M_{[0,t]} (x)}
	\nu (dx)
	\end{equation}
 	for any positive $t$.
\end{lemma}
\begin{proof}
	First of all note that, since $ \Psi^M_{t}$ is not smooth, we cannot work directly with forms and we can only argue with measures.
	Let now $ A $ be a set contained
	in a foliation chart
	$ V \simeq U \times I $ stable under
	$ \Psi^M_{t} $
	and such that in this chart
	$ A \simeq U' \times I' $.
	To prove the result
	it is enough to
	compute
	$
	(\Psi^M_{-t})_{*}
	\nu
	(A)
	$
	for all such $ A $'s, since
	such sets cover
	$ M $.

	On every plaque $ U \times \left\{ \zeta \right\} $,
	the flow $ \Psi^M_{t} $ induces a flow
	$ \Psi_{t}^{M,\zeta} $
	and we have
	$ \Psi^M_{t} (U' \times I')=
	\left\{ (\Psi_{t}^{M,\zeta}(z), \zeta);
		(z, \zeta) \in U' \times I'
	\right\}
	$.
	Recall that the measure $ \nu $
	in the chart $ V $ can be written as 
	$ f_\alpha(z, \zeta) \omega_{P}(dz)d\zeta $.
	Because $ \Psi_{t} $ is conjugated to the geodesic flow
	on $ T^{1} X $ by $ \Xi $,
	and $ \Xi $ 
	is smooth and measure preserving along the leaves,
	we have
	\[ (\Psi_{t}^{M,\zeta})^{*} \omega_{P} = e^{t} \omega_{P} .\]
	We can finally compute
	\begin{align}
		(\Psi^M_{-t})_{*}
		\nu
		(A)
		& =
		\nu ( \Psi^M_{t}(A))
		=
		\int_{\Psi^M_{t}(U' \times I')}
		f_\alpha(z, \zeta)
		\omega_{P}
		(dz)d\zeta
		\\
		&=
		\int_{I'} 
		\int_{\Psi_{t}^{M,\zeta}(U')} 
		f_\alpha(z, \zeta)
		\omega_{P} (dz) d\zeta
		\\
		&=
		\int_{I'} 
		\int_{U'} 
		f_\alpha(\Psi_{t}^{M,\zeta}(z), \zeta)
		e^{t}
		\omega_{P} (dz) d\zeta
		\\
		&=
		\int_{I'} 
		\int_{U'} 
		\frac{
			f_\alpha(\Psi_{t}^{M,\zeta}(z), \zeta)
		}
		{
			f_\alpha(z, \zeta)
		}
		e^{t}
		f_\alpha(z, \zeta)
		\omega_{P} (dz) d\zeta
		\\
		&=
		\int_{U' \times I'} 
		e^{t} 
		\abs{
			D \hol
			\Psi^M_{[0,t]}
			(z,\zeta)
		}
		f_\alpha(z, \zeta)
		\omega_{P}(dz)
		d\zeta
		,
	\end{align}
	where the last equality follows from the local expression \eqref{eq:local_hol} of the norm of the holonomy. This is exactly what we wanted to prove.
\end{proof}

We can prove now prove
\autoref{th:bound_chi_fol} about the main estimate for the transverse Lyapunov exponent.
The main idea is that the flow $\Psi^M_t$ is expanding along the leaves of $\FF_{u}^M$
while the total volume of
$ \nu $ remains constant.
\begin{proof}
	[Proof of theorem \ref{th:bound_chi_fol}]
	Using \autoref{lemma:flowmeasure} we obtain
	\begin{equation}
		1
		=
		\int_{M} 
		\nu
		=
		\int_{M} 
		(\Psi^M_{-t})_{*} \nu
		=
		\int_{M} 
		e^{t}
		\abs{D \hol \Psi^M_{[0,t]} (x)}
		\nu (dx)
		.
	\end{equation}
	Applying the logarithm function
	and using Jensen's inequality
	we get
	\begin{equation}
		\int_{M}
		( t +
		\log
		\abs{D \hol \Psi^M_{[0,t]} (x)}
		)
		\nu(dx)
		\le 0.
	\end{equation}
	Finally we can divide by $ t $ and obtain
	\begin{equation}
		1 +
		\int_{M} 
		\frac{1}{t}
		\log
		\abs{D \hol \Psi^M_{[0,t]} (x)}
		\nu(dx)
		\le 0
		.
	\end{equation}
	Now, by theorem \ref{th:def_chi} the integrand
	converges to $ \lambda_{T} $
	for $ \nu $-almost every $ x $
	and it is bounded uniformly in $ t $,
    by continuity of the holonomy.
	By the dominated convergence
	theorem, the integral converges then
	to $ \lambda_{T} $ and so we have the desired estimate
	\begin{equation}
		1 + \lambda_{T}
		\le 0.
	\end{equation}
\end{proof}

\section{Thermodynamic formalism}
\label{sec:thermo}

In this section, we recall the setting of the thermodynamic formalism
developed by
Bridgeman, Canary, Labourie and Sambarino
in
\cite{bridgemanPressureMetricAnosov2015} and 
\cite{bridgemanSimpleRootFlows2017}
(see also \cite{bridgemanHessianHausdorffDimension2020}).
In the next section, we will relate it to Lyapunov exponents and use it to prove Theorem
\ref{th:main}.

This thermodynamic formalism
is a machinery which allows
to encode quantities associated
to Anosov representations
using reparametrisations
of the geodesic flow.

We recall now the notions and results we will need for our purposes,
and refer to
\cite{bridgemanPressureMetricAnosov2015}
for all the precise definitions and proofs.

Let as above $X$ be a Riemann surface, $ T^{1} X $ be its  unit tangent bundle and $\Psi_t $ be the geodesic flow 
on it (note that in
\cite{bridgemanPressureMetricAnosov2015} and
\cite{bridgemanSimpleRootFlows2017},
they work with the more general
geodesic flow of the group $ \pi_1(X) $).
We denote by $ O $ the set of periodic
orbits of $ \Psi_t $ and, for
any $ a \in O$, we denote by
$ p(a) $ the period of $ a $.

Given a positive Hölder continuous
function $ f $ on $ T^{1} X $,
there exists a reparametrisation $ \Psi^{f} $
of the flow $ \Psi $ such that,
for every periodic orbit
$ a \in O $,
the period of $a$ for
the flow $ \Psi^{f} $
is given by
\begin{equation}
	p_{f} (a)
	=
	\int_{0}^{p(a)}
	f(\Psi_s(x))
	ds,
\end{equation}
where $ x $ is any point in $ a $.
The flow $ \Psi^{f} $ is only
Hölder continuous but it
is a Metric Anosov flow
(see \cite[sec. 3.2]{bridgemanPressureMetricAnosov2015}).

Denote by $ R_{T}(f) $
the set
$
\left\{ 
	a \in O
	\tq
	p_{f}(a)
	\le
	T
\right\}
$.
The topological
entropy
of the flow $ \Psi^{f} $
is given by
\begin{equation}
	h_{f}
	=
	\lim
	_{T \to +\infty}
	\frac{1}{T}
	\log
	\#
	R_{T}(f)
	,
\end{equation}
and it is finite and positive.
The unique probability
measure
of maximal entropy
$ \mu_{f}  $ 
for $ \Psi^{f} $
is given by
\begin{equation}
	\mu_{f}
	=
	\lim
	_{T \to +\infty}
	\frac{1}{\# R_{T}(f)}
	\sum
	_{a \in R_{T}(f)} 
	\delta_{a}^{f}
	,
\end{equation}
where $ \delta_{a}^{f} $
is the probality measure
supported on $ a $
and invariant by
the flow $ \Psi^{f} $.

Given  two
positive Hölder continuous 
functions $f$ and $g $ on $ T^{1} X $,
we define their intersection
$ I(f,g) $ by
\begin{equation}
	I(f,g)
	=
	\int_{T^1X}
	\frac{g}{f}
	d \mu_{f} 
	,
\end{equation}
and their renormalized
intersection
$ J(f,g) $ by
\begin{equation}
\label{eq:renint}
	J(f,g)
	=
	\frac{h_{g} }{h_{f} }
	I(f,g)
	.
\end{equation}
The Hessian of the renormalized intersection
was used in
\cite{bridgemanPressureMetricAnosov2015}
to define a pressure form
on the set of pressure zero
Hölder functions,
generalizing the Weil-Peterson form.

We finally recall an important estimate of the renormalized intersection.
Recall that $f$
and $g$ are Livsic cohomologous
when the flows
$ \Psi^f $
and
$ \Psi^g $
are Hölder conjugate.
In that case these two
flows have the same periods.
\begin{proposition}\cite[Prop. 3.8]{bridgemanPressureMetricAnosov2015})]
\label{prop:bound_thermo}
    The renormalized intersection
satifisfies the lower bound:
\begin{equation}
	J(f,g) \ge 1,
\end{equation}
with equality if and only if
$ h_{f} f $
and
$ h_{g} g $
are Livsic cohomologous.
\end{proposition}

\section{Lyapunov exponents and Anosov and Hitchin representations}
\label{sec:lyapexp}

In this section we recall the definition of Lyapunov exponents associated to a representation of the fundamental group of a Riemann surface and investigate the special cases of Anosov and, more specially, Hitchin representations.

\subsection{Oseledets theorem and Lyapunov exponents.}\label{subsec:Oseledets}

Let $X$ be a compact Riemann surface of genus greater than one and $T^1X$ be its unit cotangent bundle equipped with the Liouville probability measure. Let moreover $\rho:\pi_1(X)\to \SL(d, \R) $ be a representation and $E_\rho$ be the associated flat vector bundle over $T^1X$, i.e. the quotient of
$ T^1 \tilde{X} \times \R^{d} $
by the diagonal action of
$ \pi_1(X) $,
acting on the second factor by
$\rho$.
The geodesic flow $ \Psi_{t} $  on $ T^{1} X $
induces a flow $ \ti{\Psi}_{t} $  on $ E_{\rho} $ by parallel transport.
We finally equip $E_{\rho}$ with an arbitrary
measurable norm $\norm{\cdot}$ (here measurable, and in particular defined up to measure zero sets, is enough since, similarly to \autoref{rmk:recurrence}, we can use the Poincaré recurrence Theorem to work on a compact subset where the norm is defined).

We define now the Lyapunov exponents of $\rho$ with respect to $X$ by applying the theorem of Oseledets
(see e.g. \cite{arnoldRandomDynamicalSystems1998})
to the linear flow
$\ti{\Psi}_{t}$ lying over the ergodic flow
$ \Psi_{t} $.
\begin{theorem}[Oseledets]
    \label{th:oseledets}
    There exist real constants $\tilde{\lambda}_1> \dots > \tilde{\lambda}_r$ and a decomposition 
	\[E_{\rho}=\bigoplus_{i=1}^r E_{\rho}^{i}\]
	by measurable real vector bundles such that for a.e. $x\in T^1X$ and all $v\in (E_{\rho}^{i})_x\setminus \{0\}$, it holds
\[	\tilde{\lambda}_{i}=\pm\lim_{t \to +\infty} \frac{1}{t}\log
	\norm{\ti{\Psi}_{\pm t} v}.\]
 Moreover, for all $ i \neq j $, we have
\begin{equation}
    \label{eq:oseledets_distance}
	\lim_{t \to +\infty}
	\frac{1}{t}
	\log
	d(
	(E^{i}_\rho)_{\Psi_{t} x},
	(E^{j}_\rho)_{\Psi_{t} x}
	)
	=
	0
	,
\end{equation}
where $ d $ is the Hausdorff
distance on the compact subset
of the projective space.
\end{theorem} 

	The set of values $\lambda_i=\lambda_i(X,\rho)$, for $i=1,\dots,d$, obtained by considering the values $\tilde{\lambda}_i  $ \emph{repeated with multiplicity} $\dim E_\rho^{i}$, is called the set of \emph{Lyapunov exponents} or Lyapunov spectrum of $(X,\rho)$. Note that the Lyapunov exponents are independent of the choice of the norm function on $E_{\rho}$.
 \begin{remark}
 \label{rmk:lyap_def}
Lyapunov exponents can be equivalently defined as
\begin{equation}
    \lambda_i
    =
    \lim_{t \to +\infty}
    \frac{1}{t}
    \log
	\sigma_i ( \ti{\Psi}_{t} )
 ,
\end{equation}
where
$ \sigma_i $
are the singular values
(defined in
\autoref{sub:anosov_representations}).  We will use this characterization when we will relate Lyapunov exponents to other quantities in \autoref{sec:relationships_between_the_different_lyapunov_exponents}.
 \end{remark}

We finally recall that we can associate  two flags to the decomposition of $E_\rho$,
\emph{the forward flag}  $
F_\rho^{1}\subset
F_\rho^{2} \subset
\cdots
\subset
F_\rho^{r}
$ and the \emph{backward flag}
$
B_\rho^{1}\subset
B_\rho^{2} \subset
\cdots
\subset
B_\rho^{r}
$ 
defined by
\begin{itemize}
	\item 
 $ F_\rho^{i} = E_\rho^{r+1-i} \oplus \cdots \oplus E_\rho^{r} $,
	\item
	$ B_\rho^{i} =
		E_\rho^{1} \oplus \cdots \oplus E_\rho^{i} $.
\end{itemize}
These measurable bundles satisfy the following properties. For almost any $x\in T^1X$ it holds:
\begin{itemize}
	\item $ (B^{i}_\rho)_{x} \cap (F^{d+1-i}_\rho)_{x}
		= (E^{i}_\rho)_{x} $
		,
	\item 		$
		\lim_{t \to +\infty} 
		\frac{1}{t}
		\log
		\norm{
			\ti{\Psi}_{t} v
		}
		=
		\tilde{\lambda}_{r+1-i}
		$
		if and only if
		$ v \in (F^{i}_\rho)_{x} \setminus (F^{i-1}_\rho)_{x} $ 
		,
	\item 		$
		\lim_{t \to + \infty} 
		\frac{1}{t}
		\log
		\norm{
			\ti{\Psi}_{-t} v
		}
		=
		-
		\tilde{\lambda}_{i}
		$
		if and only if
		$ v \in (B^{i}_\rho)_{x} \setminus (B^{i-1}_\rho)_{x} $
		.
\end{itemize}
Note that all the measurable bundles $E_\rho^{i}$, $B^i$ anf $F^i$ are equivariant with respect to the action of $ \ti{\Psi}_{t} $.
We say that a point $ x \in T^{1} X $ is called \emph{regular} is it is a point for which the previous properties hold.

\subsection{Anosov and Hitchin representations}
\label{sub:anosov_representations}

We consider now special representations $\rho:\pi_1(X)\to \SL(d,\R)$.
The notion of Anosov representations
of fundamental groups
of hyperbolic manifolds
has been
introduced by Labourie
in
\cite{labourieAnosovFlowsSurface2006}.
Since then,
it has been generalized
to general hyperbolic groups
and different equivalent
definitions have been found
\cite{guichardAnosovRepresentationsDomains2012},
\cite{gueritaudAnosovRepresentationsProper2017},
\cite{kapovichAnosovSubgroupsDynamical2017},
\cite{kapovichMorseActionsDiscrete2014a},
\cite{bochiAnosovRepresentationsDominated2019}.
Here we state a definition
adapted to what is needed in the following.

Let $ \abs{\cdot} $
be word metric
on $ \pi_1(X) $
associated to the
choice of an arbitrary
symmetric generating set.
Fix a Euclidean norm
$ \norm{\cdot}$
on
$\SL(d, \R)$.
For a matrix $ g \in \SL(d, \R) $,
denote by
$ \sigma_{1}(g) \ge \dots \ge \sigma_{d}(g) $
its singular values, i.e. the eigenvalues of $ \sqrt{g^{*} g} $
defined using this norm.
Remark that
$  \norm{g} = \sigma_1 (g)$
and
$  \norm{\wedge^k g} = \sigma_1 (g) + \cdots + \sigma_k (g)$,
where
$ \wedge^k g$
is the automorphism
of the exterior power
$ \wedge^k \R^d$
induced by
$g$.

The following definition
is from
\cite{bochiAnosovRepresentationsDominated2019}.
For any $p=1,\dots,d$, we say that a representation
$ \rho : \pi_1 (X) \to \SL(d, \R) $ 
is  \emph{$p$-Anosov} if there exist
constants $ C,\lambda > 0 $
such that
\begin{equation}
	\frac
	{\sigma_{p+1}(\rho(\gamma))}
	{\sigma_{p}(\rho(\gamma))}
	\le
	C e^{-\lambda \abs{\gamma}} 
	,
\end{equation}
for every
$ \gamma $ 
in
$\pi_1(X)$.
This property is independent
of the word metric chosen.
Note that by substituting $\gamma^{-1}$ to $\gamma$ in the previous expression, it is easy to show that a representation
is $ p $-Anosov
if and only if it is
$ (d-p) $-Anosov
and that
$\wedge^p \rho$
is $1$-Anosov
if and only if $\rho$ is
$p$-Anosov.

Recall that the boundary $ \binf \pi_1(X) $ of the
group $ \pi_1(X)$ 
is a topological circle
with a Hölder structure.
We recall now one important property of $p$-Anosov representations
.
\begin{theorem}[{\cite[Prop. 4.9]{bochiAnosovRepresentationsDominated2019}}]
\label{thm:boundary_maps}
    Let $ \rho : \pi_1(X) \to \SL(d, \R) $  be a $ p $-Anosov representation. There exist two $ \rho $-equivariant Hölder continuous maps
$ \xi_{p}  : \binf \pi_1(X) \to G_{p} $
and
$ \xi_{d-p}  : \binf \pi_1(X) \to G_{d-p} $,
where $ G_{i} $ is the Grassmanian
of $ i $-planes in $ \R^{d} $, satisfying
\begin{equation}
	\xi_{p} (x)
	\oplus
	\xi_{d-p} (y)
	=
	\R^{d},
\end{equation}
for every
$ x \neq y $
in
$ \binf \pi_1(X)$.
\end{theorem}

The maps $\xi_{i}$  are called the \emph{boundary maps} of $ \rho $.

We are now able to recall the definition of hyperconvex representations following \cite{pozzettiConformalityRobustClass2019}. These representations are the ones for which the bound of \autoref{th:main_anosov} holds.

\begin{definition}
\label{def:hyperconvex}
Let  $ \rho:\pi_1(X) \to \SL(d, \R)$ be both $ 1 $-Anosov and $ 2 $-Anosov.
We say that $ \rho $
is $ (1,1,2) $-hyperconvex
if for every pairwise
distinct
$ x,y,z $
in
$ \binf \pi_1(X) $ it holds
\begin{equation}
	(\xi_{1}(x) \oplus \xi_{1}(y))
	\cap
	\xi_{d-2}(z)
	=
	\left\{ 0 \right\}
	,
\end{equation}
\end{definition}
Remark that we always have
$ \xi_{1}(x) \subset \xi_{2}(x) $
when $ \rho $ is
1-Anosov and 2-Anosov
\cite{guichardAnosovRepresentationsDomains2012}.
The main property of $ (1,1,2) $-hyperconvex that we will use in order to relate  the Lyapunov exponents  given by Oseldets theorem to the transverse Lyapunov exponents and to the thermodynamic formalism is the following.

\begin{theorem}[{\cite[Prop. 7.4]{pozzettiConformalityRobustClass2019}}]
\label{thm:hyperconvex}
   Let $ \rho:\pi_1(X) \to \SL(d, \R)$ be a 
$ (1,1,2) $-hyperconvex representation. Then
the image
$ \xi_{1}(\binf \pi_1(X)) $
is a $ C^{1+\holder} $
manifold
and its tangent space
at $ \xi_{1}(x) $ can be described as
\begin{equation}
	T_{\xi_{1} (x)} 
	\xi_{1}
	(\binf \pi_1(X))
	=
	T_{\xi_{1} (x) }
	\P \xi_{2}(x)
	.
\end{equation}
\end{theorem}

Note that even if the image of $ \xi_{1} $ 
is $ C^{1+\holder} $, 
the map $ \xi_1 $ itself is only Hölder.

We recall now the definition and properties of Hitchin representations, which are a special instance of Anosov representations.
Consider the connected component of the character variety
of representations
$ \pi_1(X) \to \PSL(d, \R) $ containing the Fuchsian representations
(via the irreducible embeding
$ \PSL(2, \R) \to \PSL(d, \R)) $. This component is called Hitchin component and the representations parametrized by this component are called Hitchin representations. 
They have been
introduced by Hitchin
in
\cite{hitchinLieGroupsTeichmuller1992}
and are central in the study of higher Teichm\"uller theory
(see e.g.
\cite{wienhardInvitationHigherTeichmuller2018}).
A Hitchin representation
$  \pi_1 (X) \to \PSL(d, \R) $
can be lifted to a representation
$ \rho:\pi_1 (X) \to \SL(d, \R) $
and the properties we are
interested in are 
independent of the lift,
so in the following we will
only work with representation
with values in $ \SL(d, \R) $.

We recall now the main properties of Hitchin representation that we will use.
\begin{theorem}
\label{thm:Hitchin}
Let $ \rho:\pi_1(X) \to \SL(d, \R) $ be a Hitchin representation. Then
\begin{itemize}
	\item $\rho $ is $ i $-Anosov, for every $ i=1, \dots, d-1$ (see \cite{labourieAnosovFlowsSurface2006});
	\item 	$ \wedge^{k} \rho:
		\pi_1(X) \to \SL(\wedge^{k} \R^{d}) $
		is  $ (1,1,2) $-hyperconvex for any $ k=1, \dots, d-1 $
		(see \cite[sec. 9.2]{pozzettiConformalityRobustClass2019}).
\end{itemize}
\end{theorem}

\subsection{Oseledets and Anosov representations}

Assume now that
$\rho:\pi_1(X) \to \SL(d, \R) $ is a $p$-Anosov representation.
We explain some consequences that this property has on the Lyapunov exponents given by the Theorem of Oseledets.
First we explain the relationship
between the notion of
dominated splitting
and the Anosov property.

Consider a continous flow
$ (\varphi_{t}) $
on a compact space
$ B $.
Suppose that
this flow lifts
to a flow
$ (\ti{\varphi}_{t}) $
on a bundle
$ E \to B $
over
$ B $ 
.
A splitting
$ E = U \oplus S $
of
$ E $
is said to be
dominated
for
$ (\ti{\varphi}_{t}) $ 
if there exists
constants
$ C,a > 0 $
such that:
\begin{equation}
	\frac
	{\norm{
			\ti{\varphi}_{t} v
	}}
	{\norm{
			v
	}}
	\le
	C
	\frac
	{\norm{
			\ti{\varphi}_{t} w
	}}
	{\norm{
			w
	}}
	e^{-at} 
	,
\end{equation}
for every
$ v \in S $
and
$ w \in U $,
see
\cite{bochiCharacterizationsDomination2009}.
In this case
we say that 
$ (\ti{\varphi}_{t}) $ 
admits a dominated splitting
of index
$\dim U$.

When $\rho$ is
$p$-Anosov,
we can construct such a splitting
of the bundle
$E_{\rho}$.
First we construct
a splitting of the bundle
$ T^{1} \ti{X} \times \R^{d} $:
to a point
$x \in T^{1} \ti{X}$
we associated the splitting
\begin{equation}
	\xi_{p}(x_{-\infty})
	\oplus
	\xi_{d-p}(x_{+\infty})
\end{equation}
of the fiber $ \R^d $.
By equivariance of
$ \xi_{p} $
and
$ \xi_{d-p} $,
this splitting descends
to a splitting
of
$ E_{\rho} $
that we denote
by
$ U \oplus S $.
Observe that this
splitting is invariant
under the lift
$ \ti{\Psi}_{t} $ of the geodesic flow on $T^1X$.
According to
\cite[prop. 4.6, prop. 4.9]{bochiAnosovRepresentationsDominated2019},
$ \rho $ being $ p $-Anosov
is equivalent to this splitting
being dominated for
$ (\ti{\Psi}_{t} ) $.
In particular
$ (\ti{\Psi}_{t}) $
is dominated of index
$p$.

We are now able to show how the property of the dominated splitting of $E_\rho$ implies inequalities between Lyapunov exponents.

\begin{proposition}
\label{prop:symplicity}
    Let $\rho:\pi_1(X) \to \SL(d, \R) $ be $p$-Anosov. Then 
    \[\lambda_{i} > \lambda_{i+1} \]
    for $i=p,d-p$.
     Equivalently 
     \[\dim E_{\rho}^{p}=\dim E_{\rho}^{d-p}= 1.\]
\end{proposition}
\begin{proof}
As explained above,
the flow $(\ti{\Psi_t})$
admits a dominated splitting
of index $p$
because $\rho$
is $p$-Anosov.
By
\cite[th. A]{bochiCharacterizationsDomination2009},
this implies that
$
\frac
{\sigma_{p+1}}
{\sigma_{p}}
(\ti{\Psi_t})$
uniformly
decreases to
$ 0 $
exponentially fast.

Since by \autoref{rmk:lyap_def} we have
$ \lambda_{p} =
\lim
\frac{1}{t}
\log \sigma_{p}
(\ti{\Psi_t})$,
this implies that 
$ \lambda_{p} > \lambda_{p+1} $.
As $p$-Anosov implies
$d-p$-Anosov,
we also have
$ \lambda_{d-p} > \lambda_{d-p+1} $.
\end{proof}

We finally relate the boundary maps
$ \xi_i:\binf \pi_1(X) \to G_{i} $ associated to a $i$-Anosov representation $ \rho $ given by \autoref{thm:boundary_maps}
and the forward and backward flags $F_\rho^i$ and $B_\rho^i$ given by the theorem of Oseledets (see \autoref{subsec:Oseledets}). Note that we can identify
$ \Sp^{1} $ 
and
$ \partial_{\infty} \pi_1(X) $
using the Fuchsian representation $j_X:\pi_1(X)\to \SL(2,\R)$ inducing the complex structure $X$ on a surface $S$.
Let now $ x \in T^{1} X $ be a regular point, i.e., a point for which the Oseledets decomposition and the Oseledets flags are defined, and let $ x_{+\infty}\in \Sp^{1} $ and $ x_{-\infty} \in \Sp^{1} $
be the boundary points in the future
and in the past of the geodesic
defined by $ x $.

\begin{proposition}
\label{prop:flags_relation}
Let $\rho:\pi_1(X) \to \SL(d, \R) $ be $p$-Anosov.
Then we have the relations
\[ \xi_{i} (x_{+\infty}) = (F_\rho^{i})_{x} \quad\text{and}\quad \xi_{i} (x_{-\infty}) = (B_\rho^{i})_{x} \]
for $ i=p,d-p $ and for any regular point $x\in T^1X$.
\end{proposition}

\begin{proof}

Consider the dominated splitting
$ E_{\rho} = U \oplus S$
associated to the
the representation
$\rho$
and recall that
$U_{x}  =\xi_{p}(x_{-\infty})$
and
$S_{x}  =\xi_{d-p}(x_{+\infty})$.

In the proof of
\cite[Th. A]{bochiCharacterizationsDomination2009} 
 it is shown that,
when
$ E_{\rho} = U \oplus S $
is dominated for
$ (\ti{\Psi}_{t} ) $,
then
for a generic $ x $ we have
\begin{equation}
	U_{x}  =
	(B_{\rho}^{p})_{x} 
	\text{ and }
	S_{x}  =
	(F_{\rho}^{d-p})_{x} 
	,
\end{equation}
where
$ B_{\rho} $
and
$ F_{\rho} $
are the backward
and forward flags
given by Oseledets theorem.

Combining these two facts
we have that 
\begin{equation}
	\xi_{p}(x_{-\infty})  =
	(B_{\rho}^{p})_{x} 
	\text{ and }
	\xi_{d-p}(x_{+\infty})  =
	(F_{\rho}^{d-p})_{x} 
\end{equation}
for any regular point $x\in T^1X$.
Finally, since if $ \rho $
is $ p $-Anosov it is also
$ d-p $-Anosov, we also have
\begin{equation}
	\xi_{d-p}(x_{-\infty})  =
	(B_{\rho}^{d-p})_{x} 
	\text{ and }
	\xi_{p}(x_{+\infty})  =
	(F_{\rho}^{p})_{x} 
	.
\end{equation}
\end{proof}

\section{Relationships and proofs of main gap theorem}
\label{sec:relationships_between_the_different_lyapunov_exponents}

In this section we prove the results about the gaps between Lyapunov exponents stated in \autoref{th:main} and \autoref{th:main_anosov} in two independent ways. The main idea is to relate the Lyapunov exponents given by Oseledets Theorem in the case of $ (1,1,2) $-hyperconvex representations to other invariants for which we can prove an inequality statement. The first approach relates the Lyapunov exponents to the transverse Lyapunov exponent defined in \autoref{sec:foliated_lyapunov_exponent} and uses the inequality obtained in \autoref{th:bound_chi_fol}, while in the second approach we relate them to the thermodynamic formalism described in \autoref{sec:thermo} and we use the inequality of \autoref{prop:bound_thermo}. 

Let $X$ be a Riemann surface and $\rho:\pi_1(X)\to \SL(d,\R)$ be a  (1,1,2)-hyperconvex Anosov representation. Let first  $\lambda_i(X,\rho)$ be the Lyapunov exponents given by Oseledets Theorem as in \autoref{subsec:Oseledets}.  Let then $\lambda_T(M_\rho)$ be the transverse Lyapunov exponent, as in \autoref{sec:foliated_lyapunov_exponent}, associated to the manifold $M_\rho$ defined below in \autoref{subsec:relation_foliated}. Let finally $J(\mathbf{1},f_\rho)$ be the renormalized intersection as defined in \autoref{sec:thermo} of the constant function $\mathbf{1}$ and the function $f_\rho$ asssociated to $\rho$ as defined below in \autoref{subsec:relation_thermo}.
The main relations we want to show in this section are the following.
\begin{theorem}
\label{thm:relations}
Let $\rho:\pi_1(X)\to \SL(d,\R)$ be a  (1,1,2)-hyperconvex
Anosov representation. Then we have the relation
    \[\lambda_1(X,\rho)-\lambda_2(X,\rho)=-\lambda_T(M_\rho)=J(\mathbf{1},f_\rho).\]
\end{theorem}

Using the previous \autoref{thm:relations} we are now able to prove the gap statement of \autoref{th:main_anosov}

\begin{proof}[Proof of \autoref{th:main_anosov}]
If $\rho:\pi_1(X)\to \SL(d,\R)$ is a  (1,1,2)-hyperconvex
Anosov representation, the inequality $\lambda_1(X,\rho)-\lambda_2(X,\rho)\geq 1$ follows from the relations expressed in \autoref{thm:relations} together with either the bound of \autoref{th:bound_chi_fol} for the transverse Lyapunov exponent or the bound of \autoref{prop:bound_thermo} for the renormalized intersection.
\end{proof}

The proof of \autoref{th:main} is the content of \autoref{subsec:proof_main_thm}.

\subsection{Relationship with transverse Lyapunov exponent}
\label{subsec:relation_foliated}

 As usual,  let $S$ be a compact surface and $X$ be a Riemann surface structure on $S$. We denote by $j_X:\pi_1(X)\to \SL(2, \R)$ the Fuchsian representation defining the Riemann surface structure.
The Fuchsian representation $j_X$ induces an identifications between the universal cover $ \ti{S} = \ti{X} $ and the hyperbolic plane
$ \HH^{2} $, and between the boundary $ \binf \pi_1(X) $ of the fundamental group of $X$  and the boundary at infinity  $ \Sp^{1}$ of $ \HH^{2} $.
We will consider the isomorphism 
\[f_u: T^{1} X \overset{\cong}{\longrightarrow} (\HH^{2} \times \Sp^{1} ) / \pi_1(X)\]
where $ \pi_1(X)$ acts diagonally by $ j_X $. The isomorphism is given by associating to  $ (\ti{x},v) $ in $ T^{1} \ti{X} $ the point $ (\ti{z}, \zeta)\in \HH^{2} \times \Sp^{1} $, where $\ti{z}$ is given by the uniformization map and  
$ \zeta $ is the point in the boundary $ \Sp^{1} $ reached by the geodesic defined by $ (x,v) $ when times goes to negative infinity.
With this identification, the \emph{horizontal foliation of this flat bundle is identified with the weak unstable foliation of the geodesic flow}.

Let now $\rho:\pi_1(X)\to \SL(2,\R)$ be a $(1,1,2)$-hyperconvex representation. Then by \autoref{thm:hyperconvex} we know that the image 
$ S_\rho:=\xi_{1} (\binf \pi_1(X)) $
is a
$ C^{1+\holder} $-submanifold
of
$ \P^{d-1} (\R) $.
Using $ j_X $, we identify
$ \binf \pi_1(X)$
and $ \Sp^{1} $ and so we obtain a map 
\[\xi_{1}^X:\Sp^{1}\longrightarrow S_\rho\subset \P^{d-1} (\R).\]

We finally construct the $ C^{1+\holder} $-manifold  $ M_{\rho} = (\HH^2 \times S_\rho)/\pi_1(X) $ where
$ \pi_1(X) $ acts diagonally by $ j_X \times \rho $.

The manifold $M_\rho$  is Hölder-homeomorphic to $ T^{1} X $ via the map
$ \Xi $ 
induced by the composition of $f_u$ and the map
$ (\tilde{z}, \zeta) \mapsto (\tilde{z}, \xi_1^X(\zeta)) $. We will denote by $\pi^\rho:M_\rho\to X$ the projection. The image $\FF_{u}^{M_\rho}$ of the unstable foliation under $\Xi$ is $ C^{1+\holder} $ and $\Xi$ is $C^1$
along the leaves of the foliation. 
We are then in the same setting as \autoref{subsec:notation_transverse} and we can define the transverse Lyapunov exponent associated to $M_\rho$.


Now that we have constructed $M_\rho$, we can now prove the relation stated in \autoref{thm:relations} between the Oseledets Lyapunov exponents associated to $X$ and $\rho$ and the transverse Lyapunov exponent associated to  $M_\rho$.
\begin{proposition}
	\label{prop:chi_p_2_1}
If $\rho:\pi_1(X)\to \SL(d,\R)$ is a  (1,1,2)-hyperconvex
Anosov representation, then
	$
		\lambda_{T}(M_\rho) = \lambda_{2}(X,\rho) - \lambda_{1}(X,\rho)
		.
	$
\end{proposition}
\begin{proof}
	We will use the notation of \autoref{sec:foliated_lyapunov_exponent}.  We consider the definition of the transverse Lyapunov exponent via the expression \eqref{eq:transverse}, i.e. 
	\[	\lambda_{T}
	=
	\lim_{t \to +\infty} 
	\frac{1}{t}
	\log
	\norm{
		D_{\zeta} 
		\Psi^{M_\rho}_{t,x} 
	}\]
	 for almost any $x\in M_\rho$ in every leaf, where recall that $\Psi^{M_\rho}_{t,x}$ is the projective linear transformation induced by the flow $\Psi^{M_\rho}_t$ on $M_\rho$ between the vertical fibers of $\pi^{\rho}$ and $\zeta$ is the projection of $x$ to the vertical fiber.  
	Fix a point $ x \in M_{\rho} $ 
	for which the previous limit holds.
	The point $ x \in M_{\rho} $
	corresponds to the point $ x' = \Xi^{-1}(x) $
	in $ T^{1} X $.
 
	Consider now the linear map
	$ \ti{\Psi}_{t,x'} :
	(E_{\rho})_{x'}
	\to
	(E_{\rho})_{\Psi_t(x')}
	$, where $E_\rho$ is the flat bundle associated to $\rho$ as defined in \autoref{subsec:Oseledets} and $\ti{\Psi}_t$ is the lift of the geodesic flow to $E_{\rho}$.
 
	The important remark here is that 
	$ \Psi^{M_\rho}_{t,x} $
	is the projective linear transformation between the fibers 
	$ (M_{\rho})_{\pi^\rho(x)} $ 
	and
	$ (M_{\rho})_{\pi^\rho(\Psi^M_{t}(x))} $
	induced by the restriction to $S_\rho$ of the projectivication of
	$ \ti{\Psi}_{t,x'}$, i.e.
 \[{\P\ti{\Psi}_{t,x'}}_{|S_\rho}=\Psi^{M_\rho}_{t,x}.\]
	
	Moreover, by definition, $ \zeta $
	is the point in
	$ (M_{\rho})_{\pi^\rho(x)} $ 
	corresponding to $ x $.
	By construction of $ M_{\rho} $, we have then
	\[ \zeta= \xi_1^X(x'_{-\infty}),\] where $ x'_{-\infty} \in \Sp^{1} $
is the boundary point in the past of the geodesic
defined by $ x' $.

	Recall finally that, since
	$ \rho $ is 
	$(1,1,2)$-hyperconvex,  by \autoref{prop:symplicity} and \autoref{prop:flags_relation}, we have
	\[ \xi^X_{1}(x'_{-\infty}) = (E^{1}_\rho)_{x'}\quad \text{and}\quad  \xi^X_{2}(x'_{-\infty}) = (B^{2}_\rho)_{x'}=(E^{1}_\rho)_{x'} \oplus (E^{2}_\rho)_{x'} \]
	and that moreover,  by \autoref{thm:hyperconvex}, we have 
 \[T_{\xi^X_1(s)}S_\rho=T_{\xi^X_{1}(s)} \P \xi^X_{2}(s).\]
 Putting together the previous displayed expressions, we find
    \[D_{\zeta} 
			\Psi^{M_\rho}_{t,x}=D_{\xi_1^X(x'_{-\infty})} 
			\Psi^{M_\rho}_{t,x}=D_{\P(E^{1}_\rho)_{x'}} \left({\P\ti{\Psi}_{t,x'}}_{|S_\rho}\right) =D_{\P(E^{1}_\rho)_{x'}} \left({\P\ti{\Psi}_{t,x'}}\right)_{|T_{\P(E^{1}_\rho)_{x'}} \P(B^{2}_\rho)_{x'}}\]
			
	We choose
	$ u \in (E^{1}_\rho)_{x'} $
	and
	$ v \in (E^{2}_\rho)_{x'} $
	two non-zero vectors.
	Since we can equip $E_\rho= T^1 \tilde{X} \times \R^{d}/\pi_1(X) $ 
with the measurable norm given by the constant norm on $\R^{d}$ (which is well-defined only up to a measurable zero set of discontinuity), we can apply the formula shown below in \autoref{lem:formula_derivative} and obtain 
	\begin{equation}
		\norm{
			D_{\zeta} 
			\Psi^{M_\rho}_{t,x} 
		}
		=
		\frac{\norm{\ti{\Psi}_{t,x'} u \wedge \ti{\Psi}_{t,x'} v}}
		{\norm{u \wedge v}}
		\left( 
			\frac{\norm{\ti{\Psi}_{t,x'} u}}{\norm{u}}
		\right)
		^{-2}.
	\end{equation}
	Remark that 
	$ \ti{\Psi}_{t,x'} u $ 
	is in
	$ (E^{1}_\rho)_{\Psi_t (x')}  $ 
	and
	$ \ti{\Psi}_{t,x'} v $ 
	is in
	$ (E^{2}_\rho)_{\Psi_t (x')} $.
	Recall the estimate
    \eqref{eq:oseledets_distance}
	on the angle between
	$ (E^{1}_\rho)_{\Psi_{t}(x')}  $
	and
	$ (E^{2}_\rho)_{\Psi_{t}(x')}  $
	given
	by
	Oseledets theorem
    \ref{th:oseledets}:
	\begin{equation}
	   \lim_{t \to +\infty}
	\frac{1}{t}
	\log
	d(
	(E^{1}_\rho)_{\Psi_{t} (x')},
	(E^{2}_\rho)_{\Psi_{t} (x')}
	)
	=
	0.
	\end{equation}
	The previous expression implies
	\begin{equation}
		\log
		\norm{\ti{\Psi}_{t,x'} u \wedge \ti{\Psi}_{t,x'} v}
		=
		\log (\norm{\ti{\Psi}_{t,x'} u} \norm{ \ti{\Psi}_{t,x'} v})
		+ o(t)
		,
	\end{equation}
since $d(L_1,L_2)=\norm{u \wedge v}/{\norm{u} \norm{v}}$
 for any two lines $L_1$,$L_2$
 and non-zero vectors
 $u \in L_1$
 and
 $v \in L_2$.
	We can then compute
	\begin{align}
		\log
		\norm{
			D_{\zeta} 
			\Psi^{M_\rho}_{t,x}
		}
		&=
		\log \norm{\ti{\Psi}_{t,x'} u}
		+
		\log \norm{ \ti{\Psi}_{t,x'} v}
		-
		2 \log \norm{\ti{\Psi}_{t,x'} u}
		+ o(t)
		\\
		&=
		\log \norm{ \ti{\Psi}_{t,x'} v}
		-
		\log \norm{\ti{\Psi}_{t,x'} u}
		+ o(t)
		.
	\end{align}
	Taking the limit for $t\to \infty$ of the previous expression and using that
	\[	\frac{1}{t}
	\log \norm{\ti{\Psi}_{t,x'} u}
	\to \lambda_{1}(X,\rho) \quad \text{and}\quad \frac{1}{t}
	\log \norm{\ti{\Psi}_{t,x'} v}
	\to \lambda_{2}(X,\rho)\]
	since
	$ u \in (E^{1}_\rho)_{x'} \setminus \left\{ 0 \right\} $
    and
	$ v \in (E^{2}_\rho)_{x'} \setminus \left\{ 0 \right\} $,
	we have proved the desired relation.
\end{proof}

We finally give a self-contained proof of the formula used previously to compute the norm of the derivative of a projective linear transformation.

\begin{lemma}
				\label{lem:formula_derivative}
				For a linear transformation
				$ g:\R^d\to \R^d$,
				the operator norm of the restriction of the derivative
				of $\P g$ to a line
				$\P(\ps{u,v})\cong\P^1$
				at the point $ \P(\langle u\rangle) $ is
				\begin{equation}
					\norm{D_{\P(\langle u\rangle)}
					\P g_{|\P\langle u,v\rangle}}
					=
					\frac{\norm{gu \wedge gv}}
					{\norm{u \wedge v}}
					\left( 
						\frac{\norm{gu}}{\norm{u}}
					\right)
					^{-2}
				\end{equation}
    where on the left hand side we consider the operator norm and on the right hand side the euclidean norm on $\R^d$.
			\end{lemma}
			\begin{proof}
				For any point
				$ L \in \P \R^{d} $
				(identified with a line
				in $ \R^{d} $),
				the tangent space
				of
				$ \P \R^{d} $
				at
				$ L $ 
				is canonically identified
				with
				the space of linear maps
				$ \Hom (L, \R^{d} / L) $.

				The derivative
				$
				D_{L} \P g:
				T_{L} \P \R^{d}
				\to
				\T_{gL} P \R^{d}
				$
				of this map
				at the point
				$ L $
				is canonically identified
				with
				the map
				\begin{align}
					\Hom
					(L, \R^{d}/L)
					&\to
					\Hom
					(gL, \R^{d}/gL)
					\\
					\varphi
					&\mapsto
					g
					\circ
					\varphi
					\circ
					g^{-1} 
					.
				\end{align} 
				A choice of Euclidian norm
				$ \norm{\cdot} $ 
				on $ \R^{d} $
				induces norms on
				$ L, \R^{d}/L $ 
				and
				$
				\Hom
				(L, \R^{d}/L)
				$,
				and this defines
				a metric on
				$ T \P \R^{d} $.

				A choice of
				vector
				$ u \in L $
				defines an isometry
				\begin{align}
					\Hom
					(L, \R^{d}/L)
					&\to
					L \wedge \R^{d} 
					\\
					\varphi
					&\mapsto
					\frac{u}{\norm{u}}
					\wedge \varphi(u)
					,
				\end{align} 
				where
				$ L \wedge \R^{d}
				\subset
				\wedge^{2} \R^{d}$
				is equiped with the norm
				on
				$\wedge^{2} \R^{d}$
				induced by the norm on
				$ \R^{d}$.

				For
				$ \varphi \in
				T_{L} \P \R^{d}
				=
				\Hom
				(L, \R^{d}/L)
				$
				and any
				$ u \in L \setminus \left\{ 0 \right\} $
				we have:
				\begin{equation}
					\norm{\varphi}
					_{\Hom(L,\R^{d} / L)}
					=
					\frac{
						\norm{\varphi(u)}
						_{\R^{d} / L}
						}{
						\norm{u}
					}
					=
					\frac{
						\norm{
							u \wedge
						\varphi(u)}_{\wedge^{2} \R^{d} } 
						}{
						\norm{u}^{2}
					}
					,	
				\end{equation}
				and
				\begin{equation}
					\norm{D_{L} \P g (\varphi)}
					_{\Hom(gL,\R^{d} / gL)}
					=
					\frac{
						\norm{g \varphi g^{-1}
						(g u)}_{\R^{d} / gL}
						}{
						\norm{g u}
					}
					=
					\frac{
						\norm{g u \wedge
						g \varphi(u)}_{\wedge^{2} \R^{d} } 
						}{
						\norm{g u}^{2}
					}
					.
				\end{equation}
				Now, given two
				non-collinear
				vectors
				$ u,v \in \R^{d} $,
				the tangent space
				at
				$ L = \P \ps{u} $
				of
				$ P = \P \ps{u,v} $
				is generated by
				$
				\varphi : u \mapsto v
				\in
				\Hom(L,\R^{d} / L)
				$.
				By the previous
				equalities
				we have:
				\begin{equation}
					\norm{D_{L} \P g_{|P }}
					=
					\frac{
						\norm{D_{L} \P g_{|P }  (\varphi)}
						}{
						\norm{\varphi}
					}
					=
					\frac
					{
						\norm{
							g u \wedge g v
						}
					}
					{
						\norm{
							u \wedge v
						}
					}
					\left( 
					\frac
					{
						\norm{
							g u
						}
					}
					{
						\norm{
							u
						}
					}
					\right)^{-2} 
					.
				\end{equation}
			\end{proof}

\subsection{Relationship with thermodynamic formalism} 
\label{subsec:relation_thermo}

We show now the relation of Oseledets Lyapunov exponents to the thermodynamic formalism introduced in \autoref{sec:thermo}.

For a matrix $ g \in \SL(d, \R) $,
denote by
$ \lambda_{1}(g) \ge \dots \lambda_{d}(g) $
the modulus of the eigenvalues of $ g $
and
define the weight
$ \varphi(g) =
\log \frac{\lambda_{1}(g)}{\lambda_{2}(g)} $.
\begin{remark}
    Note that one can consider different weight functions, e.g. in \cite[Cor. 1.5]{bridgemanPressureMetricAnosov2015} they consider the weight $ \varphi(g) =
\log\lambda_{1}(g) $.
\end{remark}

Now we recall how to associate a Hölder continuous
function on $T^1X$ to an Anosov representation of $\pi_1(X)$. Recall that the periodic orbits
of the geodesic flow are
in bijection with the conjugacy
classes of elements of $ \pi_1(X)$.

\begin{theorem}[{\cite{pozzettiConformalityRobustClass2019}}, Prop. B7]
\label{thm:frho}
Let $\rho:\pi_1(X)\to \SL(d,\R)$ be a $(1,2)$-Anosov representation. Then there exists a positive
Hölder continuous
function $ f_{\rho}:T^1X\to \R $
associated to
$\rho$
and the weight
$ \varphi $
such that
for all $\Psi_t$-periodic orbits $ a \in O $
and every $ \gamma \in \pi_1(X) $
associated to $ a $ we can write the period of $a$ for the reparametrized flow $\Psi^{f_\rho}$ as 
\begin{equation}
	p_{f_{\rho}}(a)
	=
	\varphi(\rho(\gamma))
	.
\end{equation}

\end{theorem}

\begin{remark}
\label{rem:fuchsian}
    Note that the constant function $\mathbf{1}$ on $T^1X$ is the same as $f_{j_X}$, i.e. the function given by the \autoref{thm:frho} from the uniformizing representation $j_X$ of $X$.
\end{remark}

Now consider a $ (1,1,2) $-hyperconvex representation $ \rho:\pi_1(X)\to \SL(d,\R) $. Since it is crucial for us to know the entropy of $f_{\rho}$, we recall the following fact.

\begin{theorem}[{\cite{pozzettiConformalityRobustClass2019}}, Cor. 9.1]
\label{thm:entropy}
Let $\rho:\pi_1(X)\to \SL(d,\R)$ be a $ (1,1,2) $-hyperconvex representation.
The entropy
$ h_{f_{\rho}} $
of the associated
flow $ \Psi^{f_{\rho} } $
satisfies
$ h_{f_{\rho}} = 1 $.
\end{theorem}

We can now prove the relation stated in \autoref{thm:relations} between the Oseledets Lyapunov exponents associated to $X$ and $\rho$ and the renormalized intersection $J(\mathbf{1},f_\rho)$ (see the definition in \autoref{eq:renint}) of the constant function $\mathbf{1}:T^1X\to\{1\}$ and $f_\rho$.

\begin{proposition}
\label{prop:J_relation}
If $\rho:\pi_1(X)\to \SL(d,\R)$ is a  (1,1,2)-hyperconvex
Anosov representation, then
    \[J(\mathbf{1},f_\rho)=\lambda_1(X,\rho)-\lambda_2(X,\rho).\]
\end{proposition}

\begin{proof}

    Note that by definition the constant function $ \mathbf{1} $
is the function associated to the geodesic flow $ \Psi_t $.
It is classical that the entropy of the geodesic
flow is one $ h_{\mathbf{1}} =1 $,
and that the measure of maximal
entropy is the Liouville measure
$ \mu_{\mathbf{1}} = v_{L} $.

Since by \autoref{thm:entropy} we have $h_{\mathbf{1}} = h_{f_{\rho} } = 1 $, the renormalized intersection is the same as the intersection, i.e.
\[ J(\mathbf{1}, f_{\rho}) = I(\mathbf{1}, f_{\rho})=\int_{T^1X} f_{\rho} d v_{L}.\]
Applying Birkhoff erdogic theorem, we obtain that 
for $v_{L}$-almost any $ x \in T^{1} X $
it holds
\begin{equation}
	\int_{T^1X} f_{\rho} d v_{L}
	=
	\lim_{T \to +\infty}
	\frac{1}{T}
	\int_{0}^{T}
	f_{\rho}(\Psi_{t}(x))
	dt
	.
\end{equation}
For every $ T > 0 $,
we approximate the geodesic segment
$ \Psi_{[0,T]}(x) $ by a closed geodesic
$ a_{T}(x) \in O $ of length  $ T + O(1) $
by closing it by a short geodesic arc
and we have
\begin{equation}
	\lim_{T \to +\infty}
	\frac{1}{T}
	\int_{0}^{T}
	f_{\rho}(\Psi_{t}(x))
	dx
	=
	\lim_{T \to +\infty}
	\frac{1}{T}
	p_{f_{\rho}}(a_{T}(x)).
\end{equation}
Moreover, if we denote by $ \gamma_{T}(x) \in \pi_1(X) $ the element corresponding to $ a_{T}(x) $, we can rewrite the previous expression using the characterizing property of $f_\rho$ given
\autoref{thm:frho} as
\begin{align}
	\lim_{T \to +\infty}
	\frac{1}{T}
	p_{f_{\rho}}(a_{T}(x))
	&=
	\lim_{T \to +\infty}
	\frac{1}{T}
	\varphi( \rho (\gamma_{T}(x))) \\
	&=
	\lim_{T \to +\infty}
	\frac{1}{T}
	\left(
	\log \lambda_{1}(\rho(\gamma_{T}(x)))
	-
	\log \lambda_{2}(\rho(\gamma_{T}(x)))
	\right).
\end{align}
Summarizing, we have show that 
\begin{equation}
	J(\mathbf{1},f_{\rho})
	=
	\lim_{T \to +\infty}
	\frac{1}{T}
	(
	\log \lambda_{1}(\rho(\gamma_{T}(x)))
	-
	\log \lambda_{2}(\rho(\gamma_{T}(x)))
	)
\end{equation}
for almost any $x\in T^1X$.

Since $ (\gamma_{T}(x)) $ is a quasi-geodesic asymptotic to the geodesic ray
$\Psi_{t}(x)$
and  $ \wedge^{2} \rho $ is 1-Anosov
if $ \rho $ is $ 2 $-Anosov, we can apply \autoref{lem:spectral_radius} and \autoref{lem:relationlyapsingular} below
to $ \rho $ and $ \wedge^{2} \rho $ and obtain 
\begin{align}
		&\quad \lim_{T \to +\infty}
		\left(
		\frac{1}{T}
		\log \lambda_{1}(\rho(\gamma_{T}(x)))
		-
		\frac{1}{T}
		\log \lambda_{2}(\rho(\gamma_{T}(x)))\right)
		\\
		&=\lim_{T \to +\infty}
		\left(
		\frac{1}{T}
		\log \sigma_{1}(\rho(\gamma_{T}(x)))
		-
		\frac{1}{T}
		\log \sigma_{2}(\rho(\gamma_{T}(x)))\right)
		\\
  &=\lambda_1(X,\rho)-\lambda_2(X,\rho).
\end{align}
where in the last line we have used that $\lambda_1(X,\wedge^{2} \rho)=\lambda_1(X,\rho)+\lambda_2(X,\rho)$.
\end{proof}

We give here self-contained proofs of the two results, which are probably well known, needed at the end of the previous proof to conclude.
Recall that a sequence
$(\gamma_n)$
in 
$ \pi_1 (X) $
is a quasi-geodesic
if
\begin{equation}
    A^{-1} n
    - C
    \le
    \abs{\gamma}
    \le
     A n
    + C
    ,
\end{equation}
for some constants
$ A > 0 $
and
$ C \in \R $.
By definition
a quasi-geodesic
$(\gamma_n)$
is asymptotic to a unique
$\gamma_{\infty}$
in
$\binf \pi_1 (X)$,
its limit point.

\begin{lemma}
	\label{lem:spectral_radius}
	Let $ \rho:\pi_1(X)\to \SL(d,\R) $
	be a 1-Anosov representation 
	and let
	$ (\gamma_{n}) $
	be a quasi-geodesic
	in $ \pi_1(X) $. Then
	\begin{equation}
		\lim_{n \to +\infty}
		\frac{1}{n}
		\log \lambda_{1}(\rho(\gamma_{n}))
		=
		\lim_{n \to +\infty}
		\frac{1}{n}
		\log \sigma_{1}(\rho(\gamma_{n})).
	\end{equation}
\end{lemma}
\begin{proof}
	First observe
	that the two limits
	exists,
	the first by a similar reasoning
	as in the first part of the proof of \autoref{prop:J_relation}
	and the second by a
	subadditive argument.
	In particular it is enough
	to prove the equality
	for a subsequence.

	For a matrix $ g \in \SL (n, \R) $ with $\sigma_2(g)<\sigma_1(g)$,
	let $ U^1(g) $ be the eigenspace
	associated to the
	greatest eigenvalue of
	$ g g^{*} $ and
	$ S_{d-1}(g) $ the sum of the eigenspaces
	associated to the
	$ d - 1 $ lowest eigenvalues
	of $ g^{*} g $.
	Let $ \delta(g) $ be half the projective
	distance between $ U^1(g) $
	and $ S_{d-1}(g) $.
	According to
	\cite[Lemma 14.14]{benoistRandomWalksReductive2016}
	if
	\begin{equation}
		\frac{\sigma_{2}(g)}{\sigma_{1}(g)} < \delta(g)^{2}
	\end{equation}
	then
	$ \delta(g) \norm{g} \le \lambda_{1}(g) $.

	Now if $ \rho $ is $ 1 $-Anosov
	and $ (\gamma_{n}) $ is a quasi-geodesic,
	by definition there exist
	constants
	$ C,a > 0 $ such that
	\begin{equation}
		\frac{\sigma_{2}(\rho(\gamma_{n}))}
		{\sigma_{1}(\rho(\gamma_{n}))}
		\le
		C e^{- a n}
		.
	\end{equation}
	Denoting by
	$ \xi_{1} $
	and
	$ \xi_{d-1} $
	the limit maps of
	$ \rho $
	and by
	$ \gamma_\infty $
	the limit point
	in
	$ \binf \pi_1(X) $
	associated
	to $ (\gamma_{n}) $,
	by \cite[Lemma 4.7]{bochiAnosovRepresentationsDominated2019} we have
	\begin{equation}
		\xi_{1}(\gamma_\infty)
		=
		\lim
		_{n \to \infty}
		U^1(\rho(\gamma_{n}))
		,
	\end{equation}
	and
	\begin{equation}
		\xi_{d-1}(\gamma_\infty)
		=
		\lim
		_{n \to \infty}
		S_{d-1}(\rho(\gamma_{n})^{-1} )
		.
	\end{equation}
	On the other hand, up to taking a subsequence, $ (\gamma_{n}^{-1}) $ converges
	to a point $ \gamma_\infty^{-1} \neq \gamma_\infty $,
    because $ \pi_1(X) $ acts on its boundary
    as a uniform convergence group.  By \autoref{thm:boundary_maps}, i.e. by the transversality of the limit maps,
	we have hence that 
	$
	\xi_{1}(\gamma_\infty)
	\notin
	\xi_{d-1}(\gamma_\infty^{-1})
	$.
	Moreover	\begin{equation}
		\xi_{d-1}( \gamma_\infty^{-1} )
		=
		\lim
		_{n \to \infty}
		S_{d-1}(\rho(\gamma_{n}))
		.
	\end{equation}
	In particular, there exists
	$ \delta > 0 $ such that for
	$ n $ large enough,
	the distance between
	$ U^1(\rho(\gamma_{n})) $
	and
	$ S_{d-1}(\rho(\gamma_{n})) $
	is greater than $ 2 \delta $.
	As
	\begin{equation}
		\frac{\sigma_{2}(\rho(\gamma_{n}))}
		{\sigma_{1}(\rho(\gamma_{n}))}
		\to 0,
	\end{equation}
	for $ n $ large enough, by the general fact recalled above we have
	\begin{equation}
		\delta \norm{\rho(\gamma_{n})}
		\le
		\lambda_{1}(\rho(\gamma_{n}))
		.
	\end{equation}
	Since  $\sigma_{1} = \norm{\cdot} $, then 
	$
	\lambda_{1}(\rho(\gamma_{n}))
	\le
	\norm{\rho(\gamma_{n})}
	$ and so
	this implies that
	\begin{equation}
		\lim_{n \to +\infty}
		\frac{1}{n}
		\log \lambda_{1}(\rho(\gamma_{n}))
		=
		\lim_{n \to +\infty}
		\frac{1}{n}
		\log \sigma_{1}(\rho(\gamma_{n})).
	\end{equation}
\end{proof}

The next lemma is classical. 
\begin{lemma}
\label{lem:relationlyapsingular}
Let $ \rho:\pi_1(X)\to \SL(d,\R) $
	be a 1-Anosov representation 
	and let
	$ (\gamma_{T}) $
    in $ \pi_1(X) $
    be the sequence
    constructed above
    associated to the
	  geodesic ray $\Psi_T (x) $
    for a generic $x \in T^1 X$ . Then
\begin{equation}
		\lim_{T \to +\infty}
		\frac{1}{T}
		\log \sigma_{1}(\rho(\gamma_{T}))
		=
		\lambda_{1}(X,\rho)
	\end{equation}
	
\end{lemma}
\begin{proof}
	Pick a $ x \in T^{1} X $
	generic both for the Birkhoff theorem
	and the Oseledets theorem.
	Since the operator norm of $\ti{\Psi_{t}}$ is by definition the same as its first singular value, by \autoref{rmk:lyap_def}  we have 
	\begin{equation}
		\lambda_{1}(\rho)
		=
		\lim_{T \to +\infty}
		\frac{1}{T}
		\log \norm{\ti{\Psi_{t}}}_{(E_\rho)_{x} }
		,
	\end{equation}
	where recall that $ \ti{\Psi_{t}} $
	is the lifted geodesic flow
	to the linear bundle $ E_\rho $  associated to
	$ \rho $ above $ T^{1} X $
	and
	$\norm{\ti{\Psi_{t}}}_{(E_\rho)_{x} }$
	is the operator norm of $ \ti{\Psi_{t}} $
	between the fibers
	$ (E_\rho)_{x} $ and $ (E_\rho)_{\Psi_{t} (x)} $.

	The operators
	$(\ti{\Psi_{t}})_{(E_\rho)_{x}}$
	between
	$ (E_\rho)_{x} $ and $ (E_\rho)_{\Psi_{t} (x)} $
	and the automorphism
	$(\ti{\Psi_{t}})_{\gamma_{T}} $
	of a fiber above the loop
	$ \gamma_{T} $
	differ by a
    composition by
    a bounded operator
    because 
    $(\gamma_{T})$
    is constructed
    by closing by a short geodesic arc
    the path
    $ \Psi_{[0,T]}(x) $,
	so we have:
	\begin{equation}
		\lambda_{1}(\rho)
		=
		\lim_{T \to +\infty}
		\frac{1}{T}
		\log \norm{\ti{\Psi_{t}}}_{(E_\rho)_{x} }
		=
		\lim_{T \to +\infty}
		\frac{1}{T}
		\log \norm{\ti{\Psi_{t}}}_{\gamma_{T}}
		.
	\end{equation}
	Since parallel transport in $E_\rho$ over $\gamma_T$ is given by the image of the monodromy
	$ \rho(\gamma_{T}) $, we can finally conclude
	\begin{align}
		\lambda_{1}(\rho)=
		\lim_{T \to +\infty}
		\frac{1}{T}
		\log \norm{\ti{\Psi_{t}}}_{\gamma_{T}}=
		\lim_{T \to +\infty}\frac{1}{T}
		\log \norm{\rho(\gamma_{T})}
	\end{align}
which proves what we want since by definition $\norm{\rho(\gamma_{T})}=\sigma_{1}(\rho(\gamma_{T}))$.
\end{proof}

\subsection{Proof of \autoref{th:main}}
\label{subsec:proof_main_thm}
In this section we can finally  explain how to deduce the inequality statement of 
\autoref{th:main}
from
\autoref{th:main_anosov} and show the rigidity statement for the extremal gaps situation using the thermodynamic formalism.

Let $ \rho:\pi_1(X)\to \SL(d,\R) $ be a Hitchin
representation.
Remark that we have the following
relationships between Lyapunov exponents:
\begin{equation}
	\lambda_{i+1}(\rho,X) - \lambda_{i} (\rho,X)
	=
	\lambda_{2}(\wedge^{i} \rho,X)
	-
	\lambda_{1}(\wedge^{i} \rho,X)
	,
\end{equation}
for every $ i = 1, \dots, d-1 $.
Since by \autoref{thm:Hitchin} we can apply \autoref{th:main_anosov}
to every wedge power $ \wedge^{i} \rho $ of a Hitchin representation, for $ i = 1, \dots, d-1 $,
we get the inequality part of  \autoref{th:main}.

We assume now
that $ \rho $ is a Hitchin
representation and that we are in the extremal gap case, i.e.
for every $ i = 1, \dots, d-1 $:
\begin{equation}
	\lambda_{i}(\rho,X) - \lambda_{i+1}(\rho,X)
	= 1
	.
\end{equation}
We will show that
$ \rho $ is conjugated
to the image of the Fuchsian representation
by the irreducible representation
$ \SL(2,\R) \to \SL(d,\R) $.

Let 
$ i \in \{ 1, \dots, d-1\} $
and let
$ \rho_{i}:= \wedge^{i} \rho $.
Since by assumption and by \autoref{prop:J_relation} we have
that
\begin{equation}
	1=\lambda_{1}(\rho_{i},X) - \lambda_{2}(\rho_i,X)
	=
	J(\mathbf{1},f_{\rho_{i}})
	,
\end{equation}
 the equality case of \autoref{prop:bound_thermo}  implies that
$ f_{\rho_{i}} $ is
Livsic cohomologous to 1.
In particular this means that the flows
$ \Psi_t^{f_{\rho_{i}} } $ and
$ \Psi_t $ have the same periods,
which by \autoref{rem:fuchsian} means that for every
$ \gamma \in \pi_1(X) $ it holds
\begin{equation}
	\frac{\lambda_{1}(j_X(\gamma))}
	{\lambda_{2}(j_X(\gamma))}=
 \frac{\lambda_{1}(\rho_{i}(\gamma))}
	{\lambda_{2}(\rho_{i}(\gamma))}
	=\frac{\lambda_{i}(\rho(\gamma))}
	{\lambda_{i+1}(\rho(\gamma))}
	,
\end{equation}
where $j_X:\pi_1(X)\to \SL(2,\R)$ is the Fuchsian uniformizing representation of $X$.
Since the product of the eigenvalues of any $\rho(\gamma)\in \SL(d,\R)$
is one, this implies
by an elementary calculation that
for every $ \gamma\in \pi_1(X) $ we have
\begin{equation}
	\lambda_{1}(\rho(\gamma))
    =
    \left(
	\frac{\lambda_{1}(j_X(\gamma))}
	{\lambda_{2}(j_X(\gamma))}
    \right)^{(d-1)/2}
    ,
\end{equation}
and it is well-known that
\begin{equation}
	\lambda_{1}(j^d_X(\gamma))
 =
    \left(
	\frac{\lambda_{1}(j_X(\gamma))}
	{\lambda_{2}(j_X(\gamma))}
    \right)^{(d-1)/2}
    ,
\end{equation}
where $j^d_X:\pi_1(X)\to \SL(2,\R) \to \SL(d,\R) $ is the image of $ j_X $
by the irreducible representation
$ \SL(2,\R) \to \SL(d,\R) $.

To conclude,
we apply
the \say{Hitchin rigidity} result
of
\cite[Cor. 5.19]{bridgemanIntroductionPressureMetrics2018}.
This result states that if two
Hitchin representations
$\rho_1$
and
$\rho_2$
satisfy
\begin{equation}
	\lambda_{1}(\rho_{1} (\gamma))
	=
	\lambda_{1}(\rho_{2}(\gamma))
\end{equation}
for every
$ \gamma \in \pi_1 (X) $,
then
$ \rho_{1} $
is conjugated to $ \rho_{2} $.
Applying this result
to
$\rho$
and
$j^d_X$
and
using the last
two equalities,
we get that
$\rho$
and
$j^d_X$
are conjugated. This concludes the proof of \autoref{th:main}.

\printbibliography
\end{document}